\documentclass[11pt,a4paper]{article}
\usepackage{amsfonts,amsgen,amstext,amsbsy,amsopn,amsfonts,amssymb,amscd}
\usepackage[leqno]{amsmath}
\usepackage[amsmath,amsthm,thmmarks]{ntheorem}
\usepackage{epsf,epsfig}
\usepackage{float}
\usepackage{dsfont}
\usepackage{ebezier,eepic}
\usepackage{color}
\usepackage{tikz}
\usepackage{multirow}
\usepackage{mathrsfs}
\usepackage{graphicx}
\usepackage{subfigure}
\usepackage{cases}
\usepackage{epstopdf}
\newtheorem{thm}{Theorem}[section]

\newtheorem{prop}[thm]{Proposition}
\newtheorem{prob}[thm]{Problem}
\newtheorem{lem}[thm]{Lemma}

\newtheorem{false statement}{False statement}
\newtheorem{cor}[thm]{Corollary}
\newtheorem{fact}[thm]{Fact}

\theoremstyle{definition}
\newtheorem{defn}[thm]{Definition}
\newtheorem{claim}[thm]{Claim}

\newtheorem{conj}[thm]{Conjecture}

\makeatletter \@addtoreset{equation}{section}

\newcommand{\ex}{{\rm ex}}
\def\hh{\mathcal{H}}

\def\hf{\mathcal{F}}
\def\hg{\mathcal{G}}

\def\ha{\mathcal{A}}
\def\hb{\mathcal{B}}

\def\hd{\mathcal{D}}

\def\ex{\mathbb{E}}

\begin{document}
\title{\bf\Large On $r$-wise $t$-intersecting uniform families}
\date{}
\author{Peter Frankl$^1$, Jian Wang$^2$\\[10pt]
$^{1}$R\'{e}nyi Institute, Budapest, Hungary\\[6pt]
$^{2}$Department of Mathematics\\
Taiyuan University of Technology\\
Taiyuan 030024, P. R. China\\[6pt]
E-mail:  $^1$frankl.peter@renyi.hu, $^2$wangjian01@tyut.edu.cn
}
\maketitle

\begin{abstract}
We consider families, $\hf$ of $k$-subsets of an $n$-set. For integers $r\geq 2$, $t\geq 1$, $\hf$ is called $r$-wise $t$-intersecting if any $r$ of its members have at least $t$ elements in common. The most natural construction of such a family is the full $t$-star, consisting of all $k$-sets containing a fixed $t$-set. In the case $r=2$ the Exact Erd\H{o}s-Ko-Rado Theorem shows that the full $t$-star is largest if $n\geq (t+1)(k-t+1)$. In the present paper, we prove that for $n\geq (2.5t)^{1/(r-1)}(k-t)+k$, the full $t$-star is largest in case of $r\geq 3$. Examples show that the exponent $\frac{1}{r-1}$ is best possible. This represents a considerable improvement on a recent result of Balogh and Linz.
\end{abstract}

\section{Introduction}

Let $[n]=\{1,\ldots,n\}$ be the standard $n$-element set. Let $2^{[n]}$ denote the power set of $[n]$ and let $\binom{[n]}{k}$ denote the collection of all $k$-subsets of $[n]$. A subset $\hf\subset \binom{[n]}{k}$ is called a {\it $k$-uniform family}.

The  central notion of this paper is that of $r$-wise $t$-intersecting.

\begin{defn}
For positive integers $r,t$, $r\geq 2$, a family $\hf\subset 2^{[n]}$ is called $r$-wise $t$-intersecting if $|F_1\cap F_2\cap \ldots \cap F_r|\geq t$ for all $F_1,F_2,\ldots,F_r\in \hf$.
\end{defn}

Let us define
\begin{align*}
&m(n,r,t)= \max\left\{|\hf|\colon \hf\subset 2^{[n]} \mbox{ is $r$-wise $t$-intersecting}\right\},\\[3pt]
&m(n,k,r,t)= \max\left\{|\hf|\colon \hf\subset \binom{[n]}{k} \mbox{ is $r$-wise $t$-intersecting}\right\}.
\end{align*}

Let us define the so-called Frankl families (cf. \cite{F77PHD})
\begin{align*}
&\ha_i(n,r,t) =\{A\subset [n]\colon A\cap [t+r i]\geq t+(r-1)i\}, \ 0\leq i\leq \frac{k-t}{r},\\[3pt]
&\ha_i(n,k,r,t) =\ha_i(n,t) \cap \binom{[n]}{k}.
\end{align*}

Since $\ha_i(n,r,t)$ consists of the sets $A$ satisfying $|[t+ri]\setminus A|\leq i$, that is, sets that leave out at most $i$ elements out of the first $t+ri$, $|A_1\cap \ldots\cap A_r\cap [t+ri]|\geq t+ri-ri\geq t$ for all $A_1,\ldots,A_r\in \ha_i(n,r,t)$.

\begin{conj}[\cite{F77PHD}]
\begin{align}
&m(n,r,t) =\max_i|\ha_i(n,r,t)|;\label{ineq-1.1}\\[3pt]
&m(n,k,r,t) =\max_i |\ha_i(n,k,r,t)|.\label{ineq-1.2}
\end{align}
\end{conj}

Let us note that for $r=2$ the statement \eqref{ineq-1.1} is a consequence of the classical Katona Theorem \cite{K64}.

\begin{thm}[The Katona Theorem \cite{K64}]
\[
m(n,2,t) =|\ha_{\lfloor \frac{n-t}{2} \rfloor}(n,2,t)|.
\]
\end{thm}

The case $r=2$ of \eqref{ineq-1.2} was a longstanding conjecture. It was proved in \cite{FFu} for a wide range and it was completely established by the celebrated Complete Intersection Theorem of Ahlswede and Khachatrain \cite{AK}.

A family $\hf\subset \binom{[n]}{k}$ is called a {\it $t$-star} if there exists $T\subset [n]$ with $|T|=t$ such that $T\subset F$ for all $F\in \hf$. The family $\{F\in \binom{[n]}{k}\colon T\subset F\}$ with some $T\in \binom{[n]}{t}$ is called a {\it full $t$-star}.

Let us recall a part of it that was proved earlier.

\begin{thm}[Exact Erd\H{o}s-Ko-Rado Theorem \cite{ekr}, \cite{F78}, \cite{W84}]\label{thm-ekr}
Let $\hf\subset \binom{[n]}{k}$ be a 2-wise $t$-intersecting family. Then for $n\geq (t+1)(k-t+1)$,
\[
|\hf| \leq \binom{n-t}{k-t}.
\]
Moreover, for $n>(t+1)(k-t+1)$ equality holds if and only if $\hf$ is the full $t$-star.
\end{thm}

Theorem \ref{thm-ekr} motivates the following question that is the central problem of the present paper: determine or estimate $n_0(k,r,t)$, the minimal integer $n_0$ such that for all $n\geq n_0$ and all $r$-wise $t$-intersecting families $\hf\subset \binom{[n]}{k}$, $|\hf|\leq |\ha_0(n,k,r,t)|= \binom{n-t}{k-t}$. Theorem \ref{thm-ekr} shows $n_0(k,2,t)=(t+1)(k-t+1)$.

Since the value $\binom{n-t}{k-t}$ is independent of $r$, it should  be clear that $n_0(k,r,t)$ is a monotone decreasing function of $r$. Thus $n_0(k,r,t)\leq n_0(k,2,t)=(t+1)(k-t+1)$.  For $t=1$ the exact value of $m(n,k,r,t)$ and  thereby $n_0(k,r,t)$ is known (cf. \cite{F76}):
\begin{align}\label{ineq-1.6}
m(n,k,r,1) = \left\{ \begin{array}{ll}
                 \binom{n-1}{k-1}, & \mbox{ if } n\geq \frac{r}{r-1}k\\[5pt]
                  \binom{n}{k}, & \mbox{ if } n< \frac{r}{r-1}k.
                \end{array}\right.
\end{align}

Recently, Balogh and Linz \cite{BL} showed that
\[
n_0(k,r,t)< (t+r-1)(k-t-r+ 3).
\]

The main result of the present paper is

\begin{thm}\label{thm:main-1}
For $r= 3,4$,
\begin{align}\label{ineq-1.3}
n_0(k,r,t) \leq    \left(2.5 t\right)^{\frac{1}{r-1}}(k-t)+k.
\end{align}
For $r\geq 5$,
\begin{align}\label{ineq-1.3}
n_0(k,r,t) \leq  \left(2t\right)^{\frac{1}{r-1}}(k-t)+k.
\end{align}
\end{thm}

Let us show that \eqref{ineq-1.3} is essentially  best possible for $t\geq  2^r-r$ and $r$ sufficiently large. Precisely,  for $t\geq 2^r-r$ we have
\[
\left(\frac{t+r}{2}\right)^{\frac{1}{r-1}}(k-t)<n_0(k,r,t) \leq  \left(2t\right)^{\frac{1}{r-1}}(k-t)+k.
\]
Let us prove the lower bound by showing that $|\ha_1(n,k,r,t)|>\binom{n-t}{k-t}$ for $n=\left(\frac{t+r}{2}\right)^{\frac{1}{r-1}}(k-t-r+2)+t+r-2$.
Note that
\[
|\ha_1(n,k,r,t)| = \binom{n-t-r}{k-t-r}+(t+r) \binom{n-t-r}{k-t-r+1}=  \binom{n-t-r}{k-t-r}\left(1+\frac{(t+r)(n-k)}{k-t-r+1}\right)
\]
and
\begin{align*}
\frac{|\ha_1(n,k,r,t)|}{\binom{n-t}{k-t}} &= \frac{(k-t)(k-t-1)\ldots(k-t-r+1)}{(n-t)(n-t-1)\ldots(n-t-r+1)}\left(1+\frac{(t+r)(n-k)}{k-t-r+1}\right)\\[3pt]
&= \frac{(k-t)(k-t-1)\ldots(k-t-r+2)}{(n-t)(n-t-1)\ldots(n-t-r+2)} \frac{(t+r)n-(k+1)(t+r-1)}{n-t-r+1}\\[3pt]
&> \left(\frac{k-t-r+2}{n-t-r+2}\right)^{r-1} \frac{(t+r)n-(k+1)(t+r-1)}{n-t-r+1}.
\end{align*}
If $t\geq 2^r-r$ then $n= \left(\frac{t+r}{2}\right)^{\frac{1}{r-1}}(k-t-r+2)+t+r-2\geq 2k-t-r+2$.
Let us assume $k\geq t+r$  (this is no real restriction, cf. Proposition \ref{prop-1.7} below). It follows that
\begin{align*}
\frac{(t+r)n-(k+1)(t+r-1)}{n-t-r+1} \geq (t+r) \frac{n-k-1+\frac{k+1}{t+r}}{n-t-r+1}> \frac{(t+r)(n-k)}{n-t-r+1}>\frac{t+r}{2}.
\end{align*}
Thus,
\begin{align*}
\frac{|\ha_1(n,k,r,t)|}{\binom{n-t}{k-t}} > \left(\frac{k-t-r+2}{n-t-r+2}\right)^{r-1} \frac{t+r}{2}= 1.
\end{align*}
Therefore for $t\geq 2^r-r$ we obtain that
\begin{align*}
n_0(k,r,t)&> \left(\frac{t+r}{2}\right)^{\frac{1}{r-1}}(k-t-r+2)+t+r-2 \\[3pt]
&> \left(\frac{t+r}{2}\right)^{\frac{1}{r-1}}(k-t)+ \left(\frac{t+r}{2}\right)^{\frac{1}{r-1}}\left(2\left(\frac{t+r}{2}\right)^{\frac{r-2}{r-1}}-r\right)\\[3pt]
&> \left(\frac{t+r}{2}\right)^{\frac{1}{r-1}}(k-t)+ \left(\frac{t+r}{2}\right)^{\frac{1}{r-1}}\left(2^{r-1}-r\right)\\[3pt]
&> \left(\frac{t+r}{2}\right)^{\frac{1}{r-1}}(k-t).
\end{align*}

Our next result determines $m(n,k,3,2)$ for $n> 2k\geq 4$.

\begin{thm}\label{thm:main-2}
For $n> 2k\geq 4$,
\begin{align}\label{ineq-thm-2}
m(n,k,3,2) =\binom{n-2}{k-2}.
\end{align}
Moreover, in case of equality $\hf$ is the full 2-star.
\end{thm}

 Let us note that Balogh and Linz \cite{BL} proved this for $n\geq 4(k-2)$ and in the much older paper \cite{FT}
 the weaker result $m(n,k,3,2) =(1+o(1)){n-2 \choose k-2}$ was established for $k<0.501 n$.

 Let us give two more numerical examples.

\begin{prop}\label{prop-main1}
For $n\geq 2k$,
\[
m(n,k,4,3) =\binom{n-3}{k-3}\mbox{ and } m(n,k,4,4) =\binom{n-4}{k-4}.
\]
\end{prop}

The next result establishes the analogue of \eqref{ineq-thm-2} for a wide range of the pair $(r,t)$.

\begin{thm}\label{thm:main-3}
Let $n\geq \max\left\{2k, \frac{t(t-1)}{2\log 2} +2t-1\right\}$ and $t\leq 2^{r-2}\log 2-2$.  Then
\begin{align}\label{ineq-5.2}
m(n,k,r,t) =\binom{n-t}{k-t}.
\end{align}
Moreover, in case of equality $\hf$ is the full $t$-star.
\end{thm}

Let us show that for $k\leq t+r-2$ the only $r$-wise $t$-intersecting family is the $t$-star.

\begin{prop}\label{prop-1.7}
Suppose that $\hg$ is an $r$-wise $t$-intersecting $k$-graph that is not a $t$-star ($|\cap \hg| <t$). Then $k\geq t+r$ or $k=t+r-1$ and $\hg\subset \binom{Y}{k}$ for some $(k+1)$-element set $Y$.
\end{prop}

\begin{proof}
We distinguish two cases.

(i) There exist $G_1,G_2\in \hg$ with $|G_1\cap G_2|\leq k-2$.

Since $\hg$ is 2-wise $t$-intersecting, we infer that $|G_1\cap G_2|\geq t$. Choose a $t$-subset $T$ of $G_1\cap G_2$.
Since $\hg$ is not a $t$-star,  there exist $G_3\in \hg$ and $x\in T$ such that $x\notin G_3$. Then $|G_1\cap G_2\cap G_3| \leq |(G_1\cap G_2)\setminus \{x\}|= k-3$. Similarly,
we can choose successively $G_4,\ldots,G_r$ to satisfy $|G_1\cap \ldots \cap G_r|\leq k-r$. This proves $k-r\geq t$, i.e., $k\geq r+t$.

(ii) $\hg$ is 2-wise $(k-1)$-intersecting.

Pick arbitrary $G_1,G_2\in \hg$ and set $Y=G_1\cup G_2$, $Z=G_1\cap G_2$. Then $|Y|=k+1$ and $|Z|=k-1$. Since $\hg$ is 2-wise $t$-intersecting and $|Z|=k-1>t$, there exists $G_3\in \hg$ with $Z\not\subset G_3$. Since $\hg$ is 2-wise $(k-1)$-intersecting, $|G_i\cap G_3|\geq k-1$, $i=1,2$. It follows that $G_3\subset Y$. Without loss of generality, assume that $Y=[k+1]$ and  $G_i=[k+1]\setminus \{i\}$, $i=1,2,3$. If there exists $G\in \hg$ with $|G\cap [k+1]|\leq k-1$. Then there exist $x,y\in [k+1]$ such that $G\subset [k+1]\setminus \{x,y\}$. Let $i\in [3]\setminus \{x,y\}$.  Then $|G\cap G_i|\leq k+1-3=k-2$, contradicting the assumption that $\hg$ is 2-wise $(k-1)$-intersecting. Thus $\hg\subset \binom{Y}{k}$.
\end{proof}

Based on Proposition \ref{prop-1.7} in the sequel we always assume that $n\geq k\geq t+r$.

As to the corresponding problem for the non-uniform case, Erd\H{o}s-Ko-Rado \cite{ekr} proved $m(n,2,1)=2^{n-1}$. Then the first author \cite{F77Bulletin} established $m(n,3,2) = 2^{n-2}$. After several partial results  the proof of the following result was concluded in \cite{F19}:
\begin{align}\label{ineq-1.4}
 m(n,r,t) = 2^{n-t} \mbox{ if and only if } t\leq 2^r-r-1.
\end{align}

We call  a family $\hf\subset \binom{[n]}{k}$ {\it non-trivial} if $\cap \{F\colon F\in \hf\} =\emptyset$. Define
\begin{align*}
&m^*(n,r,t)= \max\left\{|\hf|\colon \hf\subset 2^{[n]} \mbox{ is non-trivial $r$-wise $t$-intersecting}\right\},\\[3pt]
&m^*(n,k,r,t)= \max\left\{|\hf|\colon \hf\subset \binom{[n]}{k} \mbox{ is non-trivial $r$-wise $t$-intersecting}\right\}.
\end{align*}

\begin{thm}[Brace-Daykin-Frankl Theorem (cf. \cite{BD} for $t=1$ and \cite{F91} for $t\geq 2$)]\label{thm-bd}
For $t+r\leq n$ and  $t<2^r-r-1$,
\begin{align}\label{ineq-1.5}
m^*(n,r,t)= |\ha_1(n,r,t)| =(t+r+1) 2^{n-t-r}.
\end{align}
\end{thm}

Let us recall some notations and useful results.
For $i\in [n]$, define
\[
\hf(i) =\left\{F\setminus \{i\}\colon i\in F\in \hf\right\},\ \hf(\bar{i}) = \left\{F\colon i\notin F\in \hf\right\}.
\]
For $P\subset Q\subset [n]$, define
\[
\hf(Q)= \left\{F\setminus Q\colon Q\subset F\right\},\ \hf(P,Q)= \left\{F\setminus Q\colon  F\cap Q=P\right\}.
\]
Let $X$ be a finite set. 
For any $\hf\subset \binom{X}{k}$ and $1\leq b< k$, define the {\it $b$th shadow}  $\partial^{(b)} \hf$ as
\[
\partial^{(b)} \hf =\left\{E\in \binom{X}{k-b}\colon \mbox{there exists }F\in \hf \mbox{ such that }E\subset F\right\}.
\]
If $b=1$ then we simply write $\partial \hf$ and call it {\it the shadow} of $\hf$.
Define the {\it up shadow} $\partial^+ \hf$  as
\[
\partial^+ \hf =\left\{G\in \binom{X}{k+1}\colon \mbox{ there exists } F\in \hf \mbox{ such that  }F\subset G\right\}.
\]

Sperner  \cite{Sperner} proved the following result.
 
\begin{thm}[\cite{Sperner}]\label{thm-sperner}
 For $\hf\subset \binom{[n]}{k}$,
\begin{align}\label{ineq-sperner}
\frac{|\partial^+ \hf|}{\binom{n}{k+1}} \geq \frac{|\hf|}{\binom{n}{k}}.
\end{align}
\end{thm}

For $\ha,\hb\subset \binom{[n]}{k}$, we say that $\ha,\hb$ are {\it cross-intersecting} if $A\cap B\neq \emptyset$ for all $A\in \ha$ and $B\in \hb$. 

\begin{thm}[\cite{Hilton}]
Let $\ha,\hb\subset \binom{[n]}{k}$ be cross-intersecting. Then for $n\geq 2k$,
\begin{align}\label{ineq-1.7}
|\ha|+|\hb| \leq \binom{n}{k}.
\end{align}
\end{thm} 

We need the following version of the Kruskal-Katona Theorem.

\begin{thm}[\cite{Kruskal,Katona}]\label{thm-kk}
Let $n,k,m$ be positive integers with $k\leq m\leq n$ and let $\hf \subset \binom{[n]}{k}$  and. If $|\hf|>\binom{m}{k}$  then
\[
|\partial \hf|>\binom{m}{k-1}.
\]
\end{thm}

We also need an inequality concerning the $b$th shadow of an $r$-wise $t$-intersecting family. 

\begin{thm}[\cite{F91-2}]\label{thm-F91}
Let $\hf\subset \binom{[n]}{k}$ be an $r$-wise $t$-intersecting family. Then for $0<b\leq t$ we have
\begin{align}\label{ineq-key4}
|\partial^{(b)} \hf| \geq |\hf| \min_{0\leq i\leq \frac{k-t}{r-1}} \frac{\binom{ri+t}{i+b}}{\binom{ri+t}{i}}.
\end{align}
\end{thm}

\section{Shifting and lattice paths}

In \cite{ekr}, Erd\H{o}s, Ko and Rado introduced a very powerful tool in extremal set theory, called shifting.
For $\hf\subset \binom{[n]}{k}$ and $1\leq i<j\leq n$, define the shifting operator
$$S_{ij}(\hf)=\left\{S_{ij}(F)\colon F\in\hf\right\},$$
where
$$S_{ij}(F)=\left\{
                \begin{array}{ll}
                 F':= (F\setminus\{j\})\cup\{i\}, & \mbox{ if } j\in F, i\notin F \text{ and } F'\notin \hf; \\[5pt]
                  F, & \hbox{ otherwise.}
                \end{array}
              \right.
$$

It is well known (cf. \cite{F87}) that the shifting operator preserves the size of $\hf$ and the $r$-wise $t$-intersecting property. Thus one can apply the shifting operator to  $\hf$ when considering $m(n,k,r,t)$.

A family $\hf\subset \binom{[n]}{k}$ is called {\it shifted} if $S_{ij}(\hf)=\hf$ for all $1\leq i<j\leq n$. It is easy to show (cf. \cite{F87}) that every family can be transformed into a shifted family by applying the shifting operator repeatedly. Thus we can always assume that the family $\hf$ is shifted when detemining $m(n,k,r,t)$.

Let us define the shifting partial order.
Let $A=\{a_1,a_2,\ldots,a_k\}$ and $B=\{b_1,b_2,\ldots,b_k\}$ be two distinct $k$-sets with $a_1<a_2<\ldots<a_k$ and $b_1<b_2<\ldots<b_k$. We say that $A$ {\it precedes} $B$ in shifting partial order, denoted by $A\prec B$ if $a_i\leq b_i$ for $i=1,2,\ldots,k$.

Let us recall two properties of shifted families:

\begin{lem}[cf. \cite{F87}]
If $\hf\subset \binom{[n]}{k}$ is a shifted family, then $A\prec B$ and $B\in \hf$ always imply $A\in \hf$.
\end{lem}

\begin{lem}[\cite{F87}]\label{lem-2.4}
Let $\hf\subset  \binom{[n]}{k}$ be a shifted family. Then $\hf$ is $r$-wise $t$-intersecting if and only if
for every $F_1,\ldots,F_r\in \hf$ there exists $s$ such that
\begin{align}\label{ineq-2.5}
\sum_{1\leq i\leq r} |F_i\cap [s]| \geq (r-1) s+t.
\end{align}
\end{lem}

Note that $\sum\limits_{1\leq i\leq r} |F_i\cap [s]|\leq rs$ implies $s\geq t$ if such an $s$ exists.  For completeness let us include the proof.

\begin{proof}
First we show that if $\hf$ is $r$-wise $t$-intersecting  then for every $F_1,\ldots,F_r\in \hf$ there exists $s$ such that \eqref{ineq-2.5} holds.
Argue indirectly and suppose $F_1,F_2,\ldots,F_r$ is counter-example with $\sum_{1\leq i\leq r} \sum_{j\in F_i} j$ minimal.

Let $x$ be the $t$-th common vertex of $F_1,\ldots,F_r$. By our assumption,
\begin{align}\label{ineq-2.6}
\sum_{1\leq i\leq  r}|F_i\cap [x]| <(r-1)x+t =rt+(r-1)(x-t).
\end{align}
Note that
\[
|(F_1\cap [x])\cap (F_2\cap [x])\cap \ldots \cap (F_r\cap [x])|=t.
\]
By \eqref{ineq-2.6}, there exists $y<x$ such that $y$ is contained in at most $r-2$ of $F_1\cap [x],F_2\cap [x],\ldots,F_r\cap [x]$. Since $\hf$ is shifted, $F_1':=(F_1\setminus\{x\})\cup \{y\}\in \hf$. Then $F_1',F_2,\ldots,F_r$ is also a counter-example, contradicting the minimality of $\sum_{1\leq i\leq r} \sum_{j\in F_i} j$.

Next we show that if \eqref{ineq-2.6} holds for every $F_1,\ldots,F_r\in \hf$ then $\hf$ is $r$-wise $t$-intersecting. Indeed, suppose that  there exist $F_1,\ldots,F_r\in \hf$ with $|F_1\cap F_2\cap \ldots\cap F_r|<t$. Then for any $s\geq t$ at most $t-1$ elements in $[s]$ are contained in  $r$ of $F_1,F_2,\ldots, F_r$. It follows that
\[
\sum_{1\leq i\leq r} |F_i\cap [s]| \leq r(t-1)+(r-1)(s-t+1) \leq (r-1)s+t-1,
\]
a contradiction. Thus the lemma holds.
\end{proof}

Let $\hf\subset  \binom{[n]}{k}$ be a shifted $r$-wise $t$-intersecting family.  For any $F_1,\ldots,F_r\in \hf$, define $s(F_1,\ldots,F_r)$ to be the minimum $s$ such that
\begin{align*}
\sum_{1\leq i\leq r} |F_i\cap [s]| \geq (r-1) s+t.
\end{align*}
Set $s:=s(F_1,\ldots,F_r)$. Then we must have
\begin{align*}
\sum_{1\leq i\leq r} |F_i\cap [s]| = (r-1) s+t.
\end{align*}
Indeed, if $\sum\limits_{1\leq i\leq r} |F_i\cap [s]| \geq  (r-1) s+t+1$ then
\[
\sum_{1\leq i\leq r} |F_i\cap [s-1]| \geq  (r-1) s+t+1 -r\geq (r-1)(s-1)+t,
\]
contradicting the minimality of $s$. Set $F_1=F_2=\ldots =F_r=F$ for $F\in \hf$, we obtain $r |F\cap [s]| = (r-1) s+t$.
It follows that $\frac{s-t}{r}=:i$ is an integer.  Then $s=t+ri$ and
\[
\frac{(r-1) s+t}{r}=t+\frac{(r-1)(s-t)}{r} =t+(r-1)i.
\]
Thus $|F\cap [t+ri]|\geq t+(r-1)i$ holds and  we get the following corollary.

\begin{cor} [\cite{F87}]\label{cor-2.2}
Let $\hf\subset \binom{[n]}{k}$ be a shifted $r$-wise $t$-intersecting family. Then for every $F\in \hf$, there exists $i\geq 0$ so that $|F\cap [t+ri]|\geq t+(r-1)i$.
\end{cor}

In \cite{F78} a bijection between subsets and certain lattice paths was established. For $F\in \binom{[n]}{k}$, define  $P(F)$ to be the lattice path in the two-dimensional integer grid $\mathbb{Z}^2$ starting at origin as follows. In the $i$th step for $i=1,2,\ldots,n$,  from the current point $(x,y)$ the path $P(F)$ goes to $(x,y+1)$ if $i\in F$ and goes to $(x+1,y)$ if $i\notin F$.  Since $|F|=k$, there are exactly $k$ vertical steps. Thus the end point of $P(F)$ is $(n-k,k)$.

Let $\hf\subset \binom{[n]}{k}$ be a shifted $r$-wise $t$-intersecting family. By Corollary \ref{cor-2.2} we infer that $P(F)$ hits $y=(r-1)x+t$ for every $F\in \hf$. For $F\in \hf$, define $i(F)$ to be the minimum integer $i$ such that $|F\cap [t+ri]|=t+(r-1)i$. Define
\[
\hf_i= \left\{F\in \hf\colon  i(F)=i \right\}, i=0,1,2,\ldots, \left\lfloor \frac{k-t}{r-1}\right\rfloor.
\]
By Corollary \ref{cor-2.2}, $\hf_0,\hf_1,\ldots,\hf_{\lfloor \frac{k-t}{r-1}\rfloor}$ form  a partition of $\hf$.

The next lemma gives a universal bound ont the size of an $r$-wise $t$-intersecting family for $n\geq 2k-t$.

\begin{lem}
Let  $\hf\subset \binom{[n]}{k}$ be an $r$-wise $t$-intersecting family with $r\geq 3$ and $n\geq 2k-t$. Then
\begin{align}
|\hf|   \leq \sum_{0\leq i\leq t} \binom{t}{i}\binom{n-t}{k-t-(r-1)i}.\label{ineq-key0}
\end{align}
Moreover,
\begin{align}
&\sum_{i\geq 1} |\hf_i|   \leq \sum_{1\leq i\leq t} \binom{t}{i}\binom{n-t}{k-t-(r-1)i}.\label{ineq-key1}
\end{align}
\end{lem}

 \begin{figure}[t]
  \centering
  \includegraphics[width=0.45\textwidth]{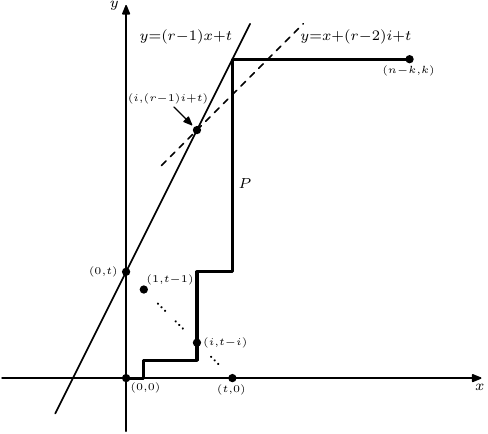}
  \caption{The lattice path $P$ goes through $(i,t-i)$ and hits the line $y=(r-1)x+t$.}\label{latticepath}
\end{figure}

\begin{proof}
Without loss of generality, we may assume that $\hf$ is shifted.
 For each $F\in \hf$, by Corollary \ref{cor-2.2} we infer that $P(F)$ hits the line $y=(r-1)x+t$. Note that the number of lattice paths that go through $(0,t)$ is exactly $\binom{n-t}{k-t}$. Let us count the number of lattice paths $P$ that do not pass $(0,t)$. Then $P$ has to go through exactly one of $(1,t-1),(2,t-2),\ldots,(t,0)$.  Since $r\geq 3$, the paths that start at $(i,t-i)$ and hit the line $y=(r-1)x+t$ have to  hit the line $y=x+(r-2)i+t$ (as shown in Figure \ref{latticepath}). Note that the number of lattice paths from $(0,0)$ to $(i,t-i)$ is $\binom{t}{i}$. By the reflection principle (cf. e.g. \cite{F78}), the number of paths from $(i,t-i)$ to $(n-k,k)$ hitting $y=x+(r-2)i+t$ equals the the number of paths from $(-(r-1)i,(r-1)i+t)$ to $(n-k,k)$, which is
  $\binom{n-t}{k-t-(r-1)i}$. Thus,
 \[
\sum_{i\geq 1} |\hf_i|  \leq \sum_{1\leq i\leq t} \binom{t}{i}\binom{n-t}{k-t-(r-1)i}.
 \]
Since $|\hf_0|\leq \binom{n-t}{k-t}$, \eqref{ineq-key0} follows.
\end{proof}

\begin{fact}\label{fact-3.1}
Suppose $\hf\subset 2^{[n]}$ is $r$-wise $t$-intersecting but $\hf$ is not a $t$-star. Then for $2\leq s<r$, $\hf$ is $s$-wise $(t+r-s)$-intersecting.
\end{fact}
\begin{proof}
Set $Y=\cap \{F\colon F\in \hf\}$. Then $|Y|<t$ and by definition $\hf(Y)$ is $r$-wise $(t-|Y|)$-intersecting and non-trivial. We need to show that $\hf(Y)$ is $s$-wise $(t-|Y|+r-s)$-intersecting. Suppose the contrary and fix $G_1,G_2,\ldots,G_s\in \hf(Y)$ satisfy $|G_1\cap \ldots \cap G_s|<t-|Y|+r-s$.

Using non-triviality we may choose successively $G_{s+1},\ldots,G_r$ to satisfy $|G_1\cap \ldots \cap G_s\cap G_{s+1}\cap \ldots\cap G_r|<t-|Y|$, i.e., $|(G_1\cup Y)\cap \ldots\cap (G_r\cap Y)|<t$, a contradiction.
\end{proof}

\begin{cor}
Let $\hf\subset \binom{[n]}{k}$ be an $r$-wise $t$-intersecting family with $r\geq 3$. If $\hf$ is not a $t$-star, then
\begin{align}\label{ineq-key2}
|\hf|   \leq \sum_{0\leq i\leq t} \binom{t}{i}\binom{n-t}{k-t-(r-1)i} - \binom{n-t-1}{k-t}.
\end{align}
\end{cor}

\begin{proof}
In the proof of \eqref{ineq-key0} we counted $\binom{n-t}{k-t}$ for the paths through $(0,t)$. Since $\hf$ is not a $t$-star, by Fact \ref{fact-3.1} we infer that  $\hf$ is $(r-1)$-wise $(t+1)$-intersecting. It follows that $\hf([t])$ is $(r-1)$-wise intersecting. By \eqref{ineq-1.6} we have $|\hf([t])| \leq \binom{n-t-1}{k-t-1}$. Now $\binom{n-t}{k-t}-\binom{n-t-1}{k-t-1} =\binom{n-t-1}{k-t}$ proves \eqref{ineq-key2}.
\end{proof}

\section{Proof of Theorem \ref{thm:main-1}}

\begin{proof}[Proof of Theorem \ref{thm:main-1}]
Let $\hf\subset \binom{[n]}{k}$ be an $r$-wise $t$-intersecting family with $n\geq (ct)^{\frac{1}{r-1}}(k-t)+k$ ($c\geq 1$ to be specified later). Without loss of generality, assume that $\hf$ is shifted and is not a $t$-star. By Fact \ref{fact-3.1}, $\hf$ is $(r-1)$-wise $(t+1)$-intersecting and  $(r-2)$-wise $(t+2)$-intersecting. It follows that
  $\hf([t])$ is $(r-1)$-wise intersecting and  $(r-2)$-wise 2-intersecting.

Since $n\geq (ct)^{\frac{1}{r-1}}(k-t)+k>2k-t$, $n-t> 2(k-t)$ follows. By \eqref{ineq-1.6} we have
\begin{align}\label{ineq-3.1}
|\hf_0|=|\hf([t])| \leq \binom{n-t-1}{k-t-1}= \frac{k-t}{n-t} \binom{n-t}{k-t} < \frac{1}{2} \binom{n-t}{k-t}.
\end{align}
If $r\geq 5$, then $\hf([t])$ is 3-wise 2-intersecting. By Theorem \ref{thm:main-2},
\begin{align}\label{ineq-3.1-2}
|\hf_0|=|\hf([t])| \leq \binom{n-t-2}{k-t-2}= \frac{(k-t)(k-t-1)}{(n-t)(n-t-1)} \binom{n-t}{k-t} < \frac{1}{4} \binom{n-t}{k-t}.
\end{align}

Using \eqref{ineq-key1} we have
\begin{align*}
|\hf| \leq   |\hf_0|+\sum_{1\leq i\leq t} \binom{t}{i}\binom{n-t}{k-t-(r-1)i}.
\end{align*}
Note that if $k-t-(r-1)i<0$ then $\binom{n-t}{k-t-(r-1)i}=0$.
Let
\[
f(n,k,r,t,i) := \binom{t}{i}\binom{n-t}{k-t-(r-1)i}.
\]
Then for $1\leq i\leq t-1$ and $n-k\geq (ct)^{\frac{1}{r-1}}(k-t)$,
\begin{align*}
\frac{f(n,k,r,t,i+1)}{f(n,k,r,t,i)}&\leq \frac{t-i}{i+1} \cdot \left(\frac{k-t-(r-1)i}{n-k+(r-1)(i+1)}\right)^{r-1}\\[3pt]
&< \frac{t}{i+1} \cdot \left(\frac{k-t}{n-k}\right)^{r-1}\\[3pt]
&\leq \frac{1}{c(i+1)}.
\end{align*}
It follows that for $i> 1$,
\[
f(n,k,r,t,i) < \frac{1}{ci}f(n,k,r,t,i-1) < \frac{1}{c^{i-1}i!}f(n,k,r,t,1).
\]
By $\sum\limits_{1\leq i\leq t}\frac{1}{c^{i}i!} <e^{1/c}-1$,
\[
\sum_{1\leq i\leq t} f(n,k,r,t,i) \leq  f(n,k,r,t,1)\sum_{1\leq i\leq t}\frac{1}{c^{i-1}i!}< t\binom{n-t}{k-t-r+1} c(e^{1/c}-1).
\]
Note that $n-k\geq (ct)^{\frac{1}{r-1}}(k-t)$ implies
\[
\frac{\binom{n-t}{k-t-r+1}}{\binom{n-t}{k-t}} < \left(\frac{k-t}{n-k}\right)^{r-1}\leq  \frac{1}{ct}.
\]
It follows that
\begin{align}\label{ineq-3.2}
\sum_{1\leq i\leq t} \binom{t}{i}\binom{n-t}{k-t-(r-1)i} =\sum_{1\leq i\leq t} f(n,k,r,t,i) <(e^{1/c}-1)\binom{n-t}{k-t}.
\end{align}
Note that $e^{1/2.5}-1<\frac{1}{2}$ and $e^{1/2}-1<\frac{3}{4}$. For $r=3,4$ and  $c=2.5$, adding \eqref{ineq-3.1} and \eqref{ineq-3.2} we get
\[
|\hf|<\frac{1}{2} \binom{n-t}{k-t}+\sum_{1\leq i\leq t} \binom{t}{i}\binom{n-t}{k-t-(r-1)i}<\frac{1}{2} \binom{n-t}{k-t}+\frac{1}{2} \binom{n-t}{k-t}=\binom{n-t}{k-t}.
\]
For $r\geq 5$ and $c=2$, adding \eqref{ineq-3.1-2} and \eqref{ineq-3.2}  we conclude that
\[
|\hf|<\frac{1}{4} \binom{n-t}{k-t}+\sum_{1\leq i\leq t} \binom{t}{i}\binom{n-t}{k-t-(r-1)i}<\frac{1}{4} \binom{n-t}{k-t}+\frac{3}{4} \binom{n-t}{k-t}=\binom{n-t}{k-t}.
\]
\end{proof}

\section{The probability of hitting the line, uniform vs non-uniform}

We need the following version of the Chernoff bound for the binomial distribution.

\begin{thm}[\cite{janson2011random}]
Let $X\in Bi(n,p)$ and $\lambda=np$. Then
\begin{align}\label{chernoff-small}
Pr(X< \lambda-a) \leq e^{-\frac{a^2}{2\lambda}}.
\end{align}
\end{thm}

We call $P(n)$ a {\it $p$-random walk of length $n$} if it starts at origin and goes up a unit with probability $p$ and goes right a unit with probability $1-p$ at each step.  Let $f(n,r,t,p)$ be the probability that a $p$-random walk $P(n)$ hits the line $y=(r-1)x+t$.  Set $f(r,t,p) =\lim\limits_{n\rightarrow \infty} f(n,r,t,p)$. That is, $f(r,t,p)$ is the probability that an infinite $p$-random walk hits the line $y=(r-1)x+t$.

\begin{lem}[\cite{F87},\cite{F91}]\label{lem-key1}
\begin{itemize}
  \item[(i)] $f(n,r,t,p)\leq f(n+1,r,t,p)$.
  \item[(ii)] $f(n+1,r,t,p) = p f(n,r,t-1,p) + (1-p) f(n,r,t+r-1,p)$.
  \item[(iii)]
  \[
      f(r,t,p)=\gamma^t,
  \]
where $\gamma$ is the unique root of  $x = p + (1-p) x^{r}$ in the open interval $(0,1)$.
  \item[(iv)] Let $\alpha_r$ be the unique root of  $x = \frac{1}{2}+ \frac{1}{2}x^{r}$. Then
  \[
      \alpha_3=\frac{\sqrt{5}-1}{2},\ \frac{1}{2}<\alpha_r< \frac{1}{2}+\frac{1}{2^r} \mbox{ for } r\geq 4.
  \]
  Moreover,
  \begin{align}\label{ineq-frankl}
  \frac{1}{2^r-r}<\alpha_r^r\leq \frac{1}{2^r-r-1} \mbox{ for } r\geq 3.
  \end{align}
\end{itemize}
\end{lem}

Let us define another type of random walk. We call $Q(n,i)$ a {\it uniform random walk} if it is chosen uniformly from all lattice paths  from $(0,0)$ to $(n-i,i)$. Let $g(n,i,r,t)$ be  the probability that a uniform random walk $Q(n,i)$ hits the line $y=(r-1)x+t$.

\begin{prop}\label{prop-key}
\begin{itemize}
  \item[(i)] $g(n,i,r,t)\leq g(n,i+1,r,t)$.
  \item[(ii)] $g(n+1,k,r,t)\leq g(n,k,r,t)$.
  \item[(iii)] For $r\geq 3$ and $t\geq 2$, $g(2k,k,r,t)\leq g(2k+2,k+1,r,t)$.
  \item[(iv)] $\lim\limits_{k\rightarrow \infty} g(2k,k,r,t)\leq f(r,t,\frac{1}{2})$.
\end{itemize}
\end{prop}

\begin{proof}
First we prove (i).  Let $\hg_i\subset \binom{[n]}{i}$ be the collection of all $i$-sets $F$ such that $P(F)$ hits the line $y=(r-1)x+t$. Let $E\in \binom{[n]}{i}$. 
If  $P(E)$ hits  $y=(r-1)x+t$ then so does $P(F)$ for every $F$ with $E\subset F$. Thus $\partial^+ \hg_i  \subset \hg_{i+1}$. Note that $g(n,i,r,t) = \frac{|\hg_i|}{\binom{n}{i}}$. By \eqref{ineq-sperner}, we conclude that
\[
g(n,i+1,r,t) = \frac{|\hg_{i+1}|}{\binom{n}{i+1}}\geq  \frac{|\partial^+ \hg_i|}{\binom{n}{i+1}}\geq \frac{|\hg_{i}|}{\binom{n}{i}} =g(n,i,r,t).
\]

Next we prove (ii). Note that a lattice path from $(0,0)$ to $(n+1-k,k)$ goes through either  $(n-k,k)$ or $(n-(k-1),k-1)$. It follows that
\[
g(n+1,k,r,t)\binom{n+1}{k} = g(n,k,r,t)\binom{n}{k}+g(n,k-1,r,t)\binom{n}{k-1}.
\]
By (i) we have $g(n,k,r,t)\geq g(n,k-1,r,t)$. Thus,
\[
g(n+1,k,r,t)\binom{n+1}{k} \leq g(n,k,r,t)\binom{n}{k}+g(n,k,r,t)\binom{n}{k-1}= g(n,k,r,t)\binom{n+1}{k}
\]
and (ii) follows.

Thirdly we prove (iii). Let $\ell(t,i)$ be the number of lattice paths from $(0,0)$ to $(i,(r-1)i+t)$ that  hit $y=(r-1)x+t$ first at $x=i$. Note that the number of lattice paths  from $(i,(r-1)i+t)$ to $(k,k)$ is $\binom{2k-ri-t}{k-(r-1)i-t}$. Thus,
\[
g(2k,k,r,t)=\sum_{0\leq i\leq \frac{k-t}{r-1}} \ell(t,i)\frac{\binom{2k-ri-t}{k-(r-1)i-t}}{\binom{2k}{k}}.
\]
Let $c_r(k,t,i) = \frac{\binom{2k-ri-t}{k-(r-1)i-t}}{\binom{2k}{k}}$. Then, using $\binom{2k}{k}/\binom{2k+2}{k+1} =\frac{k+1}{4k+2}$,
\begin{align*}
\frac{c_r(k+1,t,i)}{c_r(k,t,i)} &= \frac{\binom{2k+2-ri-t}{k+1-(r-1)i-t}}{\binom{2k-ri-t}{k-(r-1)i-t}}\cdot
\frac{\binom{2k}{k}}{\binom{2k+2}{k+1}}\\[2pt]
&= \frac{(2k+2-ri-t)(2k+1-ri-t)}{(k+1-(r-1)i-t)(k+1-i)} \cdot \frac{k+1}{4k+2}.
\end{align*}
Note that for $r\geq 3$ and $t\geq 2$ we have
\begin{align*}
&(2k+2-ri-t)(2k+1-ri-t)(k+1)-(k+1-(r-1)i-t)(k+1-i)(4k+2)\\[3pt]
=& \left( t(t-1)+2i(r-2)t+i(i(r-2)^2 - r)\right)k+ t (t - 1)+2 i (r - 1) t+i (i (r - 1)^2 - r) + i^2> 0.
\end{align*}
It follows that $c_r(k+1,t,i)> c_r(k,t,i)$ for all $0\leq i\leq \frac{k-t}{r-1}$. Thus,
\begin{align*}
g(2k,k,r,t) &=\sum_{0\leq i\leq \frac{k-t}{r-1}} \ell(t,i)c_r(k,t,i) \\[3pt]
&< \sum_{0\leq i\leq \frac{k-t}{r-1}} \ell(t,i)c_r(k+1,t,i)+\sum_{\frac{k-t}{r-1}< i\leq \frac{k+1-t}{r-1}} \ell(t,i)c_r(k+1,t,i)\\[3pt]
& =g(2k+2,k+1,r,t).
\end{align*}

Lastly we prove (iv). Let $k>  4\log k$ and let $P$ be a $p$-random walk of length $2k$ with $p=\frac{1}{2}+\sqrt{\frac{\log k}{k}}$.  Let $X$ be the number of vertical steps on $P$. Then
 \[
 \ex X=(2k)p =k+2\sqrt{k\log k}.
 \]
 Since $k\geq 4\log k$ implies $2k\geq k+2\sqrt{k\log k}$, by \eqref{chernoff-small} we have
\begin{align}\label{ineq-2}
Pr(X< k) \leq e^{-\frac{2k\log k}{k+2\sqrt{k\log k}}} \leq e^{-\log k} = \frac{1}{k}.
\end{align}
Note that
\begin{align*}
f(2k,r,t,p)&=\sum_{t\leq i\leq 2k} Pr(X=i) Pr[P\mbox{ hits } y=(r-1)x+t|X=i]\\[3pt]
& = \sum_{t\leq i\leq 2k} Pr(X=i) g(2k,i,r,t)\\[3pt]
& \geq \sum_{k\leq i\leq 2k} Pr(X=i) g(2k,i,r,t).
\end{align*}

By Proposition \ref{prop-key} (i) we have $g(2k,i,r,t)\geq g(2k,k,r,t)$ for all $i\geq k$. It follows that
\begin{align*}
f(2k,r,t,p)\geq g(2k,k,r,t)\sum_{k\leq i\leq 2k} Pr(X=i) = g(2k,k,r,t) Pr(X\geq k).
\end{align*}
By  \eqref{ineq-2}, we obtain that
\[
f\left(2k,r,t,\frac{1}{2}+\sqrt{\frac{\log k}{k}}\right) \geq g(2k,k,r,t) \frac{k-1}{k}.
\]
Letting $k$ go to infinity on both sides, we obtain that
\[
f\left(r,t,\frac{1}{2}\right)\geq \lim\limits_{k\rightarrow \infty} g(2k,k,r,t).
\]
\end{proof}

\begin{prop}\label{prop-key3}
For $n\geq 2k$,
\begin{align}\label{ineq-key3}
m(n,k,r,t)\leq \alpha_r^t \binom{n}{k},
\end{align}
where $\alpha_r$ is the unique root of  $x = \frac{1}{2}+ \frac{1}{2}x^{r}$ in the interval $(0,1)$.
\end{prop}
\begin{proof}
Let $\hf\subset \binom{[n]}{k}$ be a shifted $r$-wise $t$-intersecting family with $|\hf|=m(n,k,r,t)$. By Corollary \ref{cor-2.2}, we infer that $|\hf|$ is at most  the number of lattice paths from $(0,0)$ to $(n-k,k)$ hitting $y=(r-1)x+t$. By $n\geq 2k$ and Proposition \ref{prop-key} (ii) (iii) (iv), it follows that
\[
|\hf|=m(n,k,r,t) \leq g(n,k,r,t)\binom{n}{k} \leq g(2k,k,r,t)\binom{n}{k} \leq f\left(r,t,\frac{1}{2}\right)\binom{n}{k} =\alpha_r^t\binom{n}{k}.
\]
\end{proof}

\section{Proof of Theorem \ref{thm:main-2}}

Let us prove a useful corollary of Theorem \ref{thm-F91}.

\begin{cor}\label{cor-4.1}
Let $\hf\subset \binom{[n]}{k}$ be a $3$-wise $t$-intersecting family. If $t\geq 4$ then $|\partial^{(2)} \hf|>4|\hf|$. If $t\geq 7$ then $|\partial^{(4)} \hf|>16|\hf|$.
\end{cor}

\begin{proof}
For $t\geq 4$, we have
\[
\frac{\binom{3i+t}{i+2}}{\binom{3i+t}{i}} = \frac{(2i+t-1)(2i+t)}{(i+1)(i+2)}>2\times 2=4
\]
Applying Theorem \ref{thm-F91} with $b=2$, we obtain that
\[
|\partial^{(2)} \hf| \geq |\hf| \min_{0\leq i\leq \frac{k-t}{2}} \frac{\binom{3i+t}{i+2}}{\binom{3i+t}{i}} >4|\hf|.
\]
Similarly, if $t\geq 7$ then
\begin{align*}
\frac{\binom{3i+t}{i+4}}{\binom{3i+t}{i}} &=\frac{(2i+t)(2i+t-1)(2i+t-2)(2i+t-3)}{(i+4)(i+3)(i+2)(i+1)}\\[2pt]
&\geq \frac{(2i+7)(2i+6)(2i+5)(2i+4)}{(i+4)(i+3)(i+2)(i+1)} \\[2pt]
&=4\frac{(2i+7)(2i+5)}{(i+4)(i+1)}.
\end{align*}
Since
\[
(2i+7)(2i+5)=4\left(i^2+6i+\frac{35}{4}\right) >4(i^2+5i+4) =4(i+4)(i+1),
\]
we infer that $\frac{\binom{3i+t}{i+4}}{\binom{3i+t}{i}} >4\times 4=16$.
Applying Theorem \ref{thm-F91} with $b=4$, we obtain that
\[
|\partial^{(4)} \hf| \geq |\hf| \min_{0\leq i\leq \frac{k-t}{2}} \frac{\binom{3i+t}{i+4}}{\binom{3i+t}{i}} >16|\hf|.
\]
\end{proof}
\begin{fact}\label{fact-4.7}
For $n\geq \frac{\sqrt{4t+9}-1}{2}k$, $|\ha_1(n,k,3,t)|<\binom{n-t}{k-t}$. For $n=  \left(\frac{\sqrt{4t+9}-1}{2}-\epsilon\right)k$ with some  $0<\epsilon<\frac{1}{10}$ and $k\geq \frac{t^2+2t}{2\epsilon}$, $|\ha_1(n,k,3,t)|>\binom{n-t}{k-t}$.
\end{fact}
\begin{proof}
By Proposition \ref{prop-1.7} we assume $k\geq t+3$.
Note that  $|\ha_1(n,k,3,t)|= (t+3)\binom{n-t-3}{k-t-2}+\binom{n-t-3}{k-t-3}= \frac{(t+3)n-(t+2)(k+1)}{k-t-2} \binom{n-t-3}{k-t-3}$. Then
 \begin{align*}
 \frac{|\ha_1(n,k,3,t)|}{\binom{n-t}{k-t}}&=\frac{(k-t)(k-t-1)((t+3)n-(t+2)(k+1))}{(n-t)(n-t-1)(n-t-2)}.
 \end{align*}
 Let $n=x k$ and define
 \begin{align*}
 f(k,x)&:=(k-t)(k-t-1)((t+3)n-(t+2)(k+1))-(n-t)(n-t-1)(n-t-2)\\[3pt]
 &= (k-t)(k-t-1)\left(((t+3)x-(t+2))k-(t+2)\right)-(xk-t)(xk-t-1)(xk-t-2).
 \end{align*}
 By simplification, we obtain that
 \[
  f(k,x) = - (x-1)k \left((x^2+x-t-2)k^2+ (2t^2+4t - 3(t+1) x)k-(t+1)(t^2-t-1)\right).
 \]
 Let
 \[
 g(k,x) = (x^2+x-t-2)k^2+ (2t^2+4t - 3(t+1) x)k-(t+1)(t^2-t-1).
 \]
 If $x\geq \frac{\sqrt{4t+9}-1}{2}$, then $x^2+x-t-2\geq 0$. By $k\geq t+3$, it follows that
 \begin{align*}
  g(k,x) &\geq (x^2+x-t-2)k(t+3)+ (2t^2+4t - 3(t+1) x)k-(t+1)(t^2-t-1)\\[3pt]
  &=(t^2 + 3 (x^2-2) + t (x^2- 2 x-1))k+1 + 2 t - t^3.
 \end{align*}
 Since $x\geq \frac{\sqrt{4t+9}-1}{2}>1.56$ and $t\geq 2$ imply $3 (x^2-2)>0$ and $t (x^2- 2 x-1)+t^2>0$,
 \begin{align*}
  g(k,x) &\geq \left(t^2 + 3 (x^2-2) + t (x^2- 2 x-1)\right)(t+3)+1 + 2 t - t^3\\[3pt]
 &\geq (x^2- 2x+2)t^2 + (6x^2- 6x-7)t + 9 x^2  -17\\[3pt]
  &\geq t^2 + (6x^2- 6x-7)t + (9 x^2  -17)\\[3pt]
  &\geq (6x^2- 6x-5)t + (9 x^2  -17)\\[3pt]
  &>0 .
 \end{align*}
 Thus $f(k,x)<0$ for  $x\geq \frac{\sqrt{4t+9}-1}{2}$. Therefore $|\ha_1(n,k,3,t)|<\binom{n-t}{k-t}$ for $n\geq \frac{\sqrt{4t+9}-1}{2}k$.

 If $x= \frac{\sqrt{4t+9}-1}{2}-\epsilon$, then by $\epsilon <\frac{1}{10}$ and $t\geq 2$,
 \[
 x^2-x+2 = -\epsilon (\sqrt{4t+9}-\epsilon)\leq -\epsilon (\sqrt{17}-\epsilon)<-4\epsilon.
 \]
 It follows that for $k\geq \frac{t^2+2t}{2\epsilon}$,
 \begin{align*}
(x^2+x-t-2)k^2+ (2t^2+4t - 3(t+1) x)k-(t+1)(t^2-t-1)&< -4\epsilon k^2+(2t^2+4t) k\leq 0
 \end{align*}
 Thus $|\ha_1(n,k,3,t)|>\binom{n-t}{k-t}$ for $k\geq \frac{t^2+2t}{2\epsilon}$ and $n=  (\frac{\sqrt{4t+9}-1}{2}-\epsilon)k$.
\end{proof}

\begin{proof}[Proof of Theorem \ref{thm:main-2}]
By Proposition \ref{prop-1.7}, we assume $k\geq t+r=5$. Let $\hf\subset \binom{[n]}{k}$ be a shifted 3-wise 2-intersecting family that is not a 2-star. By Fact \ref{fact-4.7}, $|\ha_1(n,k,3,2)|<\binom{n-2}{k-2}$ for $n\geq 2k$. Thus we may assume as  well $\hf\not\subset \ha_1(n,k,3,2)$.  We partition $\hf$ according $F\cap [5]$. Define
\begin{align*}
&\hf_0 =\left\{F\in \hf\colon \{1,2\}\subset F\right\},\\[2pt]
&\hf_i= \left\{F\in  \hf\colon F\cap [5] =[5]\setminus \{i\}\right\},\ i=1,2,\\[3pt]
&\hf_3= \left\{F\in \hf\colon \{1,2\}\not\subset F \mbox{ and } |F\cap [5]|=3\right\},\\[3pt]
&\hf_4= \left\{F\in \hf\colon \{1,2\}\not\subset F \mbox{ and } |F\cap [5]|=2\right\},\\[3pt]
&\hf_5=\left\{F\in \hf\colon |F\cap [5]|=1\right\},\\[3pt]
&\hf_6=\left\{F\in \hf\colon F\cap[5]=\emptyset\right\}.
\end{align*}
Then
\[
|\hf| =\sum_{0\leq i\leq 5} |\hf_i|.
\]

Since $\hf$ is 3-wise 2-intersecting and it is not a 2-star, by Fact \ref{fact-3.1} $\hf$ is 2-wise 3-intersecting. It follows that $\hf(\{1,2\})$ is 2-wise intersecting. By \eqref{ineq-1.6} we have
\begin{align}\label{ineq-hf0}
|\hf_0| =|\hf(\{1,2\})| \leq \binom{n-3}{k-3}.
\end{align}

Set
\[
\ha_i=\left\{F\setminus [5]\colon F\cap [5]=[5]\setminus \{i\}\right\}, i=1,2.
\]
\begin{claim}\label{claim-5.11}
We may assume  that $\ha_1$ and $\ha_2$ are cross-intersecting.
\end{claim}
\begin{proof}
Indeed, otherwise there exist $F_1,F_2\in \hf$ with $F_1\cap F_2=\{3,4,5\}$. Using shiftedness and the 3-wise 2-intersecting property, $|F\cap [5]|\geq 4$ for all $F\in \hf$. Then $\hf\subset \ha_1(n,k,3,2)$, contradicting our assumption.
\end{proof}

Since $\ha_1,\ha_2\subset \binom{[6,n]}{k-4}$ are cross-intersecting, $n-5>2(k-4)$, by \eqref{ineq-1.7} we have 
\begin{align}\label{ineq-hf12}
|\hf_1|+|\hf_2|=|\ha_1|+|\ha_2| \leq \binom{n-5}{k-4}.
\end{align}

Note that $P(F)$ goes through $(2,3)$ and hits $y=2x+2$ for every $F\in \hf_3$. Since $n-5\geq 2(k-3)$, by Proposition \ref{prop-key} (ii) and (iv) the number of lattice paths from $(2,3)$ to $(n-k,k)$ hitting $y=2x+2$ is at most $ \left(\frac{\sqrt{5}-1}{2}\right)^3 \binom{n-5}{k-3}$.
The number of 3-sets $B\subset [5]$ with $[2]\not\subset B$ is $\binom{5}{3}-3=7$. Thus,
\begin{align}\label{ineq-hf3}
|\hf_3| \leq 7 \left(\frac{\sqrt{5}-1}{2}\right)^3 \binom{n-5}{k-3}<1.66\binom{n-5}{k-3}.
\end{align}

Similarly, $P(F)$ goes through $(3,2)$ and hits $y=2x+2$ for every $F\in \hf_4$ and the number of 2-sets $B\subset [5]$ with $[2]\not\subset B$ is $\binom{5}{2}-1=9$.  Since $n> 2k$ implies $n-5\geq 2(k-2)$, by Proposition  \ref{prop-key} (ii) and (iv) we infer that
\begin{align}\label{ineq-hf4}
|\hf_4| \leq 9 \left(\frac{\sqrt{5}-1}{2}\right)^6 \binom{n-5}{k-2}<0.51\binom{n-5}{k-2}.
\end{align}

Let $\hb_i = \hf(\{i\},[5])\subset \binom{[6,n]}{k-1}$, $i=1,2,\ldots,5$. By shiftedness $\hb_i$ is 3-wise 9-intersecting. Let $\hd_i=\partial^{(2)} \hb_i$. Then it is easy to see that $\hd_i$ is 3-wise $9-2\times 3=3$-intersecting. Since $n-5>2(k-3)$, by Proposition \ref{prop-key3} we infer that
\[
|\hd_i| \leq \left(\frac{\sqrt{5}-1}{2}\right)^3 \binom{n-5}{k-3}.
\]
Since $\hb_i$ is  3-wise 9-intersecting, by Corollary \ref{cor-4.1} we get
\[
|\hd_i| > 4|\hb_i|,\ i=1,2,3,4,5.
\]
Thus,
\begin{align}\label{ineq-hf5}
|\hf_5| =\sum_{1\leq i\leq 5} |\hb_i| \leq \frac{1}{4}\sum_{1\leq i\leq 5} |\hd_i| <\frac{5}{4} \left(\frac{\sqrt{5}-1}{2}\right)^3 \binom{n-5}{k-3}<0.3\binom{n-5}{k-3}.
\end{align}

Let $\ha_6=\partial^{(4)}\hf_6\subset \binom{[6,n]}{k-4}$. Since $\hf_6$ is 3-wise 12-intersecting, by  Corollary \ref{cor-4.1} we get $|\ha_6|>16|\hf_6|$.  By shiftedness, $\ha_6\subset \ha_i$ for $i=1,2$. By Claim \ref{claim-5.11}, we infer that $\ha_6$ is intersecting. Thus, by $k-2<(n-5)-(k-5)$ we obtain that
\begin{align}\label{ineq-hf6}
|\hf_6| <\frac{1}{16}|\ha_6| \leq \frac{1}{16} \binom{n-6}{k-5}< \frac{1}{16} \binom{n-5}{k-5}<\frac{1}{16} \binom{n-5}{k-2}.
\end{align}

Adding \eqref{ineq-hf0}, \eqref{ineq-hf12}, \eqref{ineq-hf3}, \eqref{ineq-hf4},
\eqref{ineq-hf5} and \eqref{ineq-hf6}, we conclude that
\begin{align*}
|\hf| &=\sum_{0\leq i\leq 6} |\hf_i| \\[3pt]
&\leq \binom{n-3}{k-3}+\binom{n-5}{k-4}+1.66\binom{n-5}{k-3}
+0.51\binom{n-5}{k-2}+0.3\binom{n-5}{k-3}+\frac{1}{16} \binom{n-5}{k-2}\\[3pt]
&< \binom{n-5}{k-5} +3\binom{n-5}{k-4}+3\binom{n-5}{k-3}+\binom{n-5}{k-2}\\[3pt]
&=\binom{n-2}{k-2}.
\end{align*}
\end{proof}

\section{Proof of Proposition \ref{prop-main1} and Theorem \ref{thm:main-3}}

Let us prove a useful inequality.

\begin{lem}
For $n> \frac{rk-t}{r-1}$,
\begin{align}\label{ineq-2.1}
& m(n,k,r,t) \leq m(n-1,k,r,t) +m(n-1,k-1,r,t).
\end{align}
\end{lem}

\begin{proof}
Let $\hf\subset \binom{[n]}{k}$ be a shifted $r$-wise $t$-intersecting family with $|\hf| = m(n,k,r,t)$. Clearly $\hf(\bar{n})$ is $r$-wise $t$-intersecting. It follows that $|\hf(\bar{n})|\leq m(n-1,k,r,t)$. We claim that $\hf(n)$ is also $r$-wise $t$-intersecting. Indeed, otherwise  there exist $G_1,G_2,\ldots,G_r\in \hf(n)$ such that $|G_1\cap G_2\cap \ldots\cap G_r|=t-1$. If each $i\in[n-1]$ is contained in at least $r-1$ of $G_1,G_2,\ldots,G_r$, then
\[
\sum_{1\leq i\leq r} G_i =rk \geq (r-1)\left((n-1)-(t-1)\right)+rt=(r-1)n+t,
\]
contradicting $n> \frac{rk-t}{r-1}$. Thus there exists $x\in [n-1]$ such that $x$ is contained in  at most $r-2$ of $G_1,G_2,\ldots,G_r$. Note that $G_i\cup \{n\}\in \hf$. Since $G_1\cup \{x\} \prec G_1\cup \{n\}$, by shiftedness we have $G_1\cup \{x\}\in \hf$. However,
\[
|(G_1\cup \{x\})\cap (G_2\cup \{n\})\cap \ldots\cap (G_r\cup \{n\})|=|G_1\cap G_2\cap \ldots\cap G_r|=t-1,
\]
contradicting the fact that $\hf$ is $r$-wise $t$-intersecting. Thus $\hf(n)$ is $r$-wise $t$-intersecting, implying that $|\hf(n)|\leq m(n-1,k-1,r,t)$. Therefore,
\[
m(n,k,r,t)=|\hf| =|\hf(\bar{n})|+|\hf(n)|\leq m(n-1,k,r,t)+m(n-1,k-1,r,t).
\]
\end{proof}

\begin{lem}\label{lem-6.2}
Suppose that $m(n,k,r,t)=\binom{n-t}{k-t}$ then
\[
m(n,k-1,r,t) =\binom{n-t}{k-1-t}.
\]
\end{lem}

\begin{proof}
Assume that $\hg\subset \binom{[n]}{k-1}$ is an $r$-wise $t$-intersecting family and $|\hg|>\binom{n-t}{k-1-t}=\binom{n-t}{n-k+1}$. Set
\[
\hg^c= \left\{[n]\setminus G\colon G\in \hg\right\}.
\]
Note that $|\hg^c|=|\hg|>\binom{n-t}{n-k+1}$. By Theorem \ref{thm-kk} we have $|\partial \hg^c|>\binom{n-t}{n-k}$. Define
\[
\hf =\left\{[n]\setminus G\colon G\in \partial \hg^c\right\}.
\]
It is easy to see that  $\hf= \partial^+(\hg)$. It follows that $\hf\subset \binom{[n]}{k}$ is $r$-wise $t$-intersecting. Then
\[
|\hf| \leq m(n,k,r,t)=\binom{n-t}{k-t},
\]
contradicting $|\hf|= |\partial \hg^c|>\binom{n-t}{n-k}=\binom{n-t}{k-t}$.
\end{proof}

Let $\hf\subset \binom{[n]}{k}$ be an $r$-wise $t$-intersecting family.
We say that $\hf$ is {\it saturated} if any addition of an extra $k$-set to $\hf$ would destroy the $r$-wise $t$-intersecting property. We say $\hf_1,\hf_2,\ldots,\hf_r\subset \binom{[n]}{k}$ are {cross $t$-intersecting} if $|F_1\cap F_2\cap\ldots \cap F_r|\geq t$ for all $F_1\in \hf_1$, $F_2\in \hf_2$, $\ldots$, $F_r\in \hf_r$.

\begin{lem}\label{lem-6.4}
Let $\hf\subset \binom{[n]}{k}$ be a shifted and saturated $r$-wise $t$-intersecting family. Let $\hg_i=\hf([t+1]\setminus \{i\},[t+1])$, $i=1,2,3,\ldots,t$. If $\hf$ is not a $t$-star, then $\hg_i=\hg_j$ for all $1\leq i<j\leq t$.
\end{lem}

\begin{proof}
Since $\hf$ is not a $t$-star, there exists some $F_0\in \hf$ such that $|F_0\cap [t]|\leq t-1$. By shiftedness, we may assume that $F_0\cap [t]= [t]\setminus \{t\}$.
By shiftedness again,
\begin{align}\label{ineq-6.2}
\hg_1\subset\hg_2\subset \ldots\subset \hg_{t}.
\end{align}
Since $F_0\setminus [t+1] \in \hg_t$, we have $\hg_t\neq \emptyset$.

By \eqref{ineq-6.2} it suffices to show that $\hg_1=\hg_t$. Suppose for contradiction that $\hg_1\subsetneq \hg_t$. Then there exists some $G_t\in \hg_t\setminus \hg_1$. Then $F:= G_t\cup ([t+r]\setminus \{t\})\in \hf$ and $F':=G_t\cup ([t+1]\setminus \{1\})\notin \hf$. By saturatedness and Lemma \ref{lem-2.4}, we infer that there exist $F_1,F_2,\ldots,F_{r-1}\in \hf$ such that for all  $x\geq 0$,
\[
|F'\cap [x]|+\sum_{1\leq i\leq r-1} |F_i\cap [x]| \leq  (r-1) x+t-1.
\]
Since $F,F_1,F_2,\ldots,F_{r-1}\in \hf$,
by Lemma \ref{lem-2.4} there exists some $s\geq t$ such that
\[
\left|F\cap [s]\right|+\sum_{1\leq i\leq r-1} |F_i\cap [s]| \geq  (r-1) s+t.
\]
It follows that $|F'\cap [s]| <|F\cap [s]|$, contradicting the fact that $s\geq t$. Thus $\hg_1=\hg_t$ and the lemma follows.
\end{proof}

\begin{lem}\label{lem-6.5}
For $k\geq 3$,
\[
m(2k,k,4,3)=\binom{n-3}{k-3}.
\]
\end{lem}

\begin{proof}
Let $n=2k$ and let $\hf\subset \binom{[n]}{k}$ be a shifted 4-wise 3-intersecting family. Without loss of generality, assume that $\hf$ is saturated and is not a 3-star. We distinguish two cases.

\vspace{3pt}
{\noindent\bf Case A.} There exist $F_1,F_2,F_3\in \hf$ such that $|F_1\cap F_2\cap F_3|=4$.
\vspace{3pt}

By shiftedness, we may assume $F_1\cap F_2\cap F_3=[4]$. Then the 4-wise 3-intersecting property implies $|F\cap [4]|\geq 3$ for all  $F\in \hf$. Define $\hh_i=\hf([4]\setminus \{i\}, [4])$ for $i=1,2,3$. By Lemma \ref{lem-6.4} these three families are identical. Set $\hh=\hh_1$. Since $\hf$ is 4-wise 3-intersecting, $\hh_1,\hh_2,\hh_3$  are cross 3-intersecting. Thus $\hh$ is 3-wise 3-intersecting. As $n-4>2(k-3)$, by Proposition \ref{prop-key3},
\begin{align}\label{ineq-6.5}
|\hh_1|+|\hh_2|+|\hh_3| =3|\hh| <3\left(\frac{\sqrt{5}-1}{2}\right)^3\binom{n-4}{k-3}<\binom{n-4}{k-3}.
\end{align}

Since $\hf$ is not a 3-star, by Fact \ref{fact-3.1}, $\hf([3])$ is 3-wise intersecting. By \eqref{ineq-1.6},
\begin{align}\label{ineq-6.6}
|\hf([3])| \leq \binom{n-4}{k-4}.
\end{align}
Adding \eqref{ineq-6.5} and \eqref{ineq-6.6}, $|\hf|<\binom{n-3}{k-3}$ follows.

\vspace{3pt}
{\noindent\bf Case B.} $\hf$ is 3-wise 5-intersecting.
\vspace{3pt}

By Proposition \ref{prop-key3},
\[
|\hf| \leq \left(\frac{\sqrt{5} -1}{2}\right)^5 \binom{2k}{k} <0.0902\binom{2k}{k}.
\]
Note that
\[
\frac{\binom{2k-3}{k-3}}{\binom{2k}{k}} =\frac{k(k-1)(k-2)}{2k(2k-1)(2k-2)} =\frac{1}{4} \times \frac{k-2}{2k-1}.
\]
Since we may assume $k\geq  4+3=7$,
\[
|\hf|<0.0902\binom{2k}{k} \leq 0.0902 \times 4 \times \frac{2k-1}{k-2}\binom{2k-3}{k-3} \leq 0.0902 \times 4 \times \frac{13}{5}\binom{2k-3}{k-3}<\binom{2k-3}{k-3}.
\]
Thus $m(2k,k,4,3)=\binom{2k-3}{k-3}$.
\end{proof}

\begin{lem}\label{lem-6.6}
For $k\geq 4$,
\[
m(2k,k,4,4)=\binom{n-4}{k-4}.
\]
\end{lem}

\begin{proof}
Let $n=2k$ and let $\hf\subset \binom{[n]}{k}$ be a shifted 4-wise 4-intersecting family. Without loss of generality, assume that $\hf$ is saturated and is not a 4-star. We distinguish three cases.

\vspace{3pt}
{\noindent\bf Case A.} $\hf$ is 3-wise 5-intersecting but not 3-wise 6-intersecting.
\vspace{3pt}

By shiftedness, we may assume $F_1\cap F_2\cap F_3=[5]$ for some $F_1,F_2,F_3\in \hf$. Then $|F\cap [5]|\geq 4$ for all  $F\in \hf$. Define $\hh_i=\hf([5]\setminus \{i\}, [5])$ for $i=1,2,3,4$. By Lemma \ref{lem-6.4} these four families are identical. Since $\hf$ is 4-wise 4-intersecting, $\hh_1,\hh_2,\hh_3,\hh_4$  are cross 4-intersecting. Thus $\hh$ is 4-wise 4-intersecting. As $n-5>2(k-4)$, by Proposition \ref{prop-key3},
\begin{align}\label{ineq-6.7}
\sum_{1\leq i\leq 4}|\hh_i| =4|\hh| <4\alpha_4^4\binom{n-5}{k-4}\overset{\eqref{ineq-frankl}}
{<}\frac{4}{2^4-4-1}\binom{n-5}{k-4}<\binom{n-5}{k-4}.
\end{align}

Since $\hf$ is not a 4-star, by Fact \ref{fact-3.1}, $\hf([4])$ is 3-wise intersecting. By \eqref{ineq-1.6},
\begin{align}\label{ineq-6.8}
|\hf([4])| \leq \binom{n-5}{k-5}.
\end{align}
Adding \eqref{ineq-6.7} and \eqref{ineq-6.8}, $|\hf|<\binom{n-4}{k-4}$ follows.

\vspace{3pt}
{\bf Case B.} $\hf$ is 3-wise 6-intersecting but not 3-wise 7-intersecting.
\vspace{3pt}

Then $\hf([4])$ is $3$-wise 2-intersecting. Since $n-4>2(k-4)$, by Theorem \ref{thm:main-2} we have $|\hf[4]|\leq \binom{n-6}{k-6}$. Fix $H_1,H_2,H_3\in \hf$ with $H_1\cap H_2\cap H_3=[6]$. Then the 4-wise 4-intersecting property implies $|F\cap [6]|\geq 4$ for all $F\in \hf$. Let
\[
\hf_i= \{F\in \hf\colon [4]\not\subset F,\ |F\cap [6]|=i\}, \ i=4,5.
\]
For  $B\in \binom{[6]}{4}$ with $B\neq [4]$ and $F\in \hf(B,[6])$, $P(F)$ is a lattice path from $(0,0)$ to $(n-k,k)$ that goes through $(2,4)$ and hits $y=3x+4$. By Proposition  \ref{prop-key} (ii) and (iv)  we infer that
\begin{align*}
|\hf_4| = \sum_{B\in \binom{[6]}{4},\ B\neq [4] }|\hf(B,[6])| <14 \alpha_4^6\binom{n-6}{k-4}\overset{\eqref{ineq-frankl}}
{<}\frac{14}{(2^4-4-1)^{3/2}}\binom{n-6}{k-4}<\binom{n-6}{k-4}.
\end{align*}

Note that $\hf_5= \cup_{1\leq i\leq 4} \hf([6]\setminus \{i\},[6])$. Let $\hg_i=\hf([6]\setminus \{i\},[6])$, $i=1,2,3,4$. By Lemma \ref{lem-6.4}, $\hg_i=\hg_j$ for $1\leq i<j\leq 4$. Since $\hg_1,\hg_2,\hg_3,\hg_4$ are cross 2-intersecting, $\hg_1$ is 4-wise 2-intersecting. Since $n-6>2(k-5)$, by Theorem \ref{thm:main-2}
\[
|\hf_5| = \sum_{1\leq i\leq 4}|\hg_i| \leq 4\binom{n-8}{k-7} = 4\frac{(k-5)(k-6)}{(n-6)(n-7)}\binom{n-6}{k-5}<\binom{n-6}{k-5}.
\]
Thus,
\[
|\hf| \leq |\hf([4])|+|\hf_4|+|\hf_5|< \binom{n-6}{k-6}+\binom{n-6}{k-4}+2\binom{n-6}{k-5}=\binom{n-4}{k-4}.
\]

\vspace{3pt}
{\bf Case C.} $\hf$ is 3-wise 7-intersecting.
\vspace{3pt}

If $\hf$ is 3-wise 8-intersecting, then we may assume $k\geq 3+8= 11$ and by Proposition \ref{prop-key3},
\begin{align*}
|\hf|\leq \left(\frac{\sqrt{5}-1}{2}\right)^8\binom{n}{k} &<\left(\frac{\sqrt{5}-1}{2}\right)^8 \left(\frac{2k-3}{k-3}\right)^4\binom{n-4}{k-4}\\[3pt]
&<\left(\frac{\sqrt{5}-1}{2}\right)^8 \left(\frac{19}{8}\right)^4\binom{n-4}{k-4}<\binom{n-4}{k-4}.
\end{align*}
Thus there exist $F_1,F_2,F_3\in \hf$ such that $F_1\cap F_2\cap F_3=[7]$. Then $|F\cap [7]|\geq 4$ for all $F\in \hf$. Since $\hf([4])$ is 3-wise 3-intersecting, by Proposition \ref{prop-key3},
\[
|\hf([4])| <\left(\frac{\sqrt{5}-1}{2}\right)^3\binom{n-4}{k-4} < 0.24 \binom{n-4}{k-4}.
\]
Let
\[
\hf_i= \{F\in \hf\colon [4]\not\subset F,\ |F\cap [7]|=i\}, \ i=4,5,6
\]
and let $\hg_i=\hf([7]\setminus \{i\},[7])$. Then by Lemma \ref{lem-6.4}, $\hg_1=\hg_2=\hg_3=\hg_4=:\hg$. Moreover, $\hg$ is 4-wise intersecting. Thus,
\[
|\hf_6| = 4|\hg| \leq 4\binom{n-8}{k-7}< 3.4\times\frac{k-6}{n-7}\binom{n-7}{k-6}+0.6\binom{n-7}{k-7}<1.8\binom{n-7}{k-6}+0.6\binom{n-7}{k-7}.
\]
Note that $P(F)$ goes though $(7-i,i)$ and hits $y=3x+4$  for each $F\in \hf_i$, $i=4,5$.
Using Proposition \ref{prop-key} (ii) and (iv), we  have
\[
|\hf_5| = 18\alpha_4^5 \binom{n-7}{k-5}< 18\times\frac{1}{2^4-4-1} \alpha_4 \binom{n-7}{k-5}=2\alpha_4 \binom{n-7}{k-5}<1.8\binom{n-7}{k-5}
\]
and
\[
|\hf_4| \leq 34 \times \alpha_4^9 \binom{n-7}{k-4} <\frac{34}{(2^4-4-1)^2} \times \left(\frac{1}{2}+\frac{1}{2^4}\right)\binom{n-7}{k-4}=\frac{17}{72}\binom{n-7}{k-4}<0.6\binom{n-7}{k-4}.
\]
Thus,
\begin{align*}
|\hf| &=|\hf([4])|  +|\hf_4|+|\hf_5|+|\hf_6| \\[3pt]
&<0.24\binom{n-4}{k-4}+0.6\left(\binom{n-7}{k-4}+3\binom{n-7}{k-5}+3\binom{n-7}{k-6}+\binom{n-7}{k-7}\right)\\[3pt]
&=0.84\binom{n-4}{k-4}<\binom{n-4}{k-4}.
\end{align*}
\end{proof}

\begin{proof}[Proof of Proposition \ref{prop-main1}]
Let $(r,t)\in \{(4,3),(4,4)\}$. By Lemmas \ref{lem-6.5} and \ref{lem-6.6}, we infer $m(2k,k,r,t)=\binom{n-t}{k-t}$.
For $n\geq 2k$, if $n$ is even then by $m(n,n/2,r,t)=\binom{n-t}{n/2-t}$ and  Lemma \ref{lem-6.3}  we have
$m(n,k,r,t)=\binom{n-t}{k-t}$.
If $n$ is odd, then $n\geq 2k$ implies $n-1\geq 2k$.
Using \eqref{ineq-2.1} we conclude that
\begin{align*}
m(n,k,r,t) \leq m(n-1,k,r,t) +m(n-1,k-1,r,t)  =\binom{n-t}{k-t}.
\end{align*}
\end{proof}

\begin{lem}\label{lem-6.3}
 If $k\geq  \frac{t(t-1)}{4\log 2} +t-1$ and  $t\leq 2^{r-2}\log 2 -2$,then
\begin{align}\label{ineq-5.3}
m(2k,k,r,t) =\binom{n-t}{k-t}.
\end{align}
Moreover, in case of equality $\hf$ is the full $t$-star.
\end{lem}
\begin{proof}
Let $n=2k$ and let $\hf\subset \binom{[n]}{k}$ be a shifted and saturated $r$-wise $t$-intersecting family. If there exist $F_1,F_2,\ldots,F_{r-1}\in \hf$ with $|F_1\cap F_2\cap \ldots \cap F_{r-1}|=t$, then $\hf$ is a $t$-star and \eqref{ineq-5.3} follows. Thus we may assume that $\hf$ is and $(r-1)$-wise $(t+1)$-intersecting.  We distinguish two cases.

\vspace{3pt}
{\noindent\bf Case 1.} $\hf$ is $(r-1)$-wise $(t+2)$-intersecting.
\vspace{3pt}

Then by  \eqref{ineq-key3} we have
\[
 |\hf|\leq \alpha_{r-1}^{t+2}\binom{n}{k}<\left(\frac{1}{2}+\frac{1}{2^{r-1}}\right)^{t+2}\binom{n}{k}\leq \left(\frac{1}{2}+\frac{1}{2^{r-1}}\right)^{t+2}\frac{n(n-1)\ldots(n-t+1)}{k(k-1)\ldots(k-t+1)}\binom{n-t}{k-t}.
\]
Since $n=2k$, by $t\leq 2^{r-2}\log 2 -2$ and $k\geq  \frac{t(t-1)}{4\log 2} +t-1$ we infer that
\begin{align*}
\frac{1}{2^{t+1}}\left(1+\frac{1}{2^{r-2}}\right)^{t+2}\frac{(2k-1)\ldots(2k-t+1)}{(k-1)\ldots(k-t+1)}
\leq \  & \frac{1}{4} e^{\frac{t+2}{2^{r-2}}}\prod_{1\leq i\leq t-1}\left(1+\frac{i}{2(k-i)}\right)\\[2pt]
<\   &\frac{1}{4} \exp\left(\frac{t+2}{2^{r-2}}+\sum_{1\leq i\leq t-1} \frac{i}{2(k-t+1)}\right)\\[2pt]
\leq\   &\frac{1}{4} \exp\left(\log 2+ \frac{t(t-1)}{4(k-t+1)}\right)\\[2pt]
\leq\   &\frac{1}{4} \exp\left(\log 2+ \log 2\right)= 1.
\end{align*}
Thus $|\hf| < \binom{n-t}{k-t}$.

\vspace{3pt}
{\noindent\bf Case 2.} There exist $F_1,\ldots,F_{r-1}\in \hf$ with $|F_1\cap \ldots \cap F_{r-1}|=t+1$.
\vspace{3pt}

By shiftedness, we may assume $F_1\cap \ldots \cap F_{r-1}=[t+1]$.
Then the $r$-wise $t$-intersecting property implies $|F\cap [t+1]|\geq t$ for all $F\in \hf$.  Let $\hg_i= \hf([t+1]\setminus\{i\},[t+1])$ for $i=1,2,\ldots,t+1$. By Lemma \ref{lem-6.4} we infer $\hg_1=\hg_2=\ldots=\hg_t$.
Let $\hg=\hg_i$, $i=1,2,\ldots,t$. Then
\[
|\hf| =|\hf([t])|+ t|\hg|.
\]
 By Fact \ref{fact-3.1},
  $\hf([t])$ is $(r-1)$-wise intersecting. Using \eqref{ineq-1.6}, we obtain that
  \[
  |\hf([t])| \leq \binom{n-t-1}{k-t-1} <\frac{k-t}{2k-t}\binom{n-t}{k-t}<\frac{1}{2}\binom{n-t}{k-t}.
  \]
  We are left to show $t|\hg|\leq \frac{1}{2}\binom{n-t}{k-t}$.

If $t\geq r$, then $\hg_1,\hg_2,\ldots,\hg_r$ is cross $(r-1)$-intersecting on $[t+2,n]$. Since $\hg_1=\hg_2=\ldots=\hg_r$ , $\hg$ is $r$-wise $(r-1)$-intersecting. Note that  $t\leq 2^{r-2}\log 2 -2\leq \frac{2^r-r-1}{4}$ holds for all $r\geq 3$. By Proposition \ref{prop-key3},
\[
|\hg| \leq \alpha_r^{r-1}\binom{n-t-1}{k-t} \overset{\eqref{ineq-frankl}}{<}\frac{1}{\alpha_r(2^r-r-1)}\binom{n-t-1}{k-t} <\frac{1}{2t} \binom{n-t}{k-t}
\]
and we are done.

By Fact \ref{fact-3.1},  $\hf$ is $t$-wise $r$-intersecting. 
If $r>t$ then $\hg_1,\hg_2,\ldots,\hg_t$ is cross $(r-1)$-intersecting on $[t+2,n]$. Since $\hg_1=\hg_2=\ldots=\hg_t$,    $\hg$ is $t$-wise $(r-1)$-intersecting. By Theorem \ref{thm:main-2} we may assume $t\geq 3$. Then
\[
|\hg| \leq \alpha_t^{r-1}\binom{n-t-1}{k-t} \leq \alpha_t^{t}\binom{n-t-1}{k-t} \overset{\eqref{ineq-frankl}}{<}\frac{1}{2^t-t-1}\binom{n-t-1}{k-t}\leq  \frac{1}{2t} \binom{n-t}{k-t}.
\]
Thus $t|\hg|\leq \frac{1}{2}\binom{n-t}{k-t}$ and the lemma is proven.
\end{proof}

\begin{proof}[Proof of Theorem \ref{thm:main-3}]
Note that $n\geq \frac{t(t-1)}{2\log 2} +2t-1$ implies
\begin{align}\label{ineq-6.1}
\frac{n}{2}>\frac{n-1}{2}\geq \frac{t(t-1)}{4\log 2} +t-1.
\end{align}
If $n$ is even, then by applying Lemma \ref{lem-6.3} we infer that
\[
m\left(n,\frac{n}{2},r,t\right) =\binom{n-t}{\frac{n}{2}-t}.
\]
Since $\frac{n}{2}\geq k$, by Lemma \ref{lem-6.2} we have
\begin{align*}
m\left(n,k,r,t\right) =\binom{n-t}{k-t}.
\end{align*}

If $n$ is odd, then $n\geq 2k$ implies $n\geq 2k+1$. By \eqref{ineq-6.1} and applying Lemma \ref{lem-6.3},
\[
m\left(n-1,\frac{n-1}{2},r,t\right) =\binom{n-1-t}{\frac{n-1}{2}-t}.
\]
Since $\frac{n-1}{2}\geq k>k-1$, by Lemma \ref{lem-6.2}
\[
m(n-1,k,r,t) = \binom{n-1-t}{k-t}\mbox{ and }m(n-1,k-1,r,t) = \binom{n-1-t}{k-1-t}.
\]
Using \eqref{ineq-2.1} we conclude that
\begin{align*}
m(n,k,r,t) \leq m(n-1,k,r,t) +m(n-1,k-1,r,t)  =\binom{n-t}{k-t}.
\end{align*}
\end{proof}

\section{Concluding remarks}

The area of research concerning $r$-wise $t$-intersecting non-uniform families is quite large and there are several results we could not even mention. The case of uniform families, that is, adding a new parameter $k$, increases this variety. In the present paper we stayed mostly in the range $k\leq \frac{1}{2}n$. However, it is completely legitimate to consider the range $k\sim cn$ for any fixed $c<1$ as long as $c\leq \frac{r-1}{r}$.

If one wants to extend the results to such a range it seems to be essential to answer the following question.

\begin{prob}
Let $c<\frac{r-1}{r}$ and denote by $p(n,k,r,t)$ the probability that a random lattice path from $(0,0)$ to $(n-k,k)$ hits the line $y=(r-1)x+t$. Let $\alpha$ be the unique root of $c-x+(1-c)x^r=0$ in $(0,1)$. Does the inequality
\begin{align}
p(n,k,r,t) <\alpha^t \mbox{ holds always if }k\leq cn?
\end{align}
\end{prob}

It seems to be rather difficult to determine the exact value of $n_0(k,r,t)$. Based on Fact \ref{fact-4.7}, let us make the following:

\begin{conj}
For $n\geq  \frac{\sqrt{4t+9}-1}{2}k$,
\begin{align*}
m(n,k,3,t) =\binom{n-t}{k-t}.
\end{align*}
\end{conj}

Another  important problem would be to determine $m^*(n,k,r,1)$, the uniform version of the Brace-Daykin Theorem (the case $t=1$ of Theorem \ref{thm-bd}). In the case $r=2$ the solution is given by the Hilton-Milner Theorem \cite{HM}.

Let us recall the Hilton-Milner-Frankl Theorem. Define
\begin{align*}
&\hb(n,k,r,t) =\left\{B\in \binom{[n]}{k}\colon [t+r-2]\subset B,\  B\cap [t+r-1,k+1]\neq \emptyset\right\}\\[2pt]
&\qquad\qquad\qquad\qquad\cup \left\{[k+1]\setminus \{j\}\colon 1\leq j\leq t+r-2\right\}.
\end{align*}

\begin{thm}[Hilton-Milner-Frankl Theorem \cite{HM,F78-2,AK0}]
For $n\geq (k-t+1)(t+1)$,
\begin{align}\label{ineq-hmfrankl}
m^*(n,k,2,t)=\max\left\{|\ha_1(n,k,2,t)|,|\hb(n,k,2,t)|\right\}.
\end{align}
\end{thm}

Note that both families $\ha_1(n,k,2,t)$ and $\hb(n,k,2,t)$ are $r$-wise $(t+2-r)$-intersecting, in particular, $(t+1)$-wise 1-intersecting. Thus in the range $(k-t+1)(t+1)<n$, i.e.,
$k<\frac{n}{t+1}+t-1$,
\[
m^*(n,k,r,t+2-r) =m^*(n,k,2,t).
\]
However the case $k\sim cn$ with $\frac{1}{t+1}<c<\frac{r-1}{r}$
 appears to be much harder. In \cite{FW} the following was proved.

\begin{thm}[\cite{FW}]
Let $0<\varepsilon<\frac{1}{10}$. For $n\geq \frac{4}{\varepsilon^2}+7$ and  $\left(\frac{1}{2}+\varepsilon \right)n\leq k\leq \frac{3n}{5}-3$,
\[
m^*(n,k,3,1) = |\ha_1(n,k,3,1)|.
\]
\end{thm}

\end{document}